\documentclass[a4paper,11pt]{amsart}
\usepackage[utf8]{inputenc}
\usepackage{amsxtra}
\usepackage{amsopn}
\usepackage{amsmath,amsthm,amssymb}
\usepackage{amscd}
\usepackage{amsfonts}
\usepackage{latexsym}
\usepackage{verbatim}

\usepackage{color}

\theoremstyle{plain}
\newtheorem{theorem}{Theorem}[section]
\newtheorem*{theorem*}{Theorem}
\newtheorem{definition}[theorem]{Definition}
\newtheorem{lemma}[theorem]{Lemma}
\newtheorem{prop}[theorem]{Proposition}
\newtheorem{cor}[theorem]{Corollary}
\newtheorem{rem}[theorem]{Remark}

\newtheorem*{mt*}{Main Theorem}
\DeclareMathOperator{\Imm}{Im}
\DeclareMathOperator{\Ker}{Ker}
\newcommand\C{{\mathbb C}}

\newcommand\R{{\mathbb R}}
\newcommand\Z{{\mathbb Z}}
\newcommand\T{{\mathbb T}}

\newcommand{\Cpf}{$\mathcal{C}^\infty$-pure-and-full}

\newcommand{\Cf}{$\mathcal{C}^\infty$-full}
\newcommand{\Cp}{$\mathcal{C}^\infty$-pure}

\newcommand{\del}{\partial}
\newcommand{\delbar}{\overline{\del}}
\title{On the Anti-Invariant Cohomology of Almost Complex Manifolds}
\author{Richard Hind and Adriano Tomassini}
\address{Department of Mathematics \\
University of Notre Dame \\
Notre Dame, IN 46556}
\email{hind.1@nd.edu}
\address{Dipartimento di Scienze Matematiche, Fisiche e Informatiche\\
Unit\`a di Matematica e Informatica,
Universit\`{a} degli Studi di Parma\\
Parco Area delle Scienze 53/A, 43124 \\
Parma, Italy}
\email{adriano.tomassini@unipr.it}

\thanks{The first author is partially supported by Simons Foundation grant \# 633715. \newline The second author is partially supported by the Project PRIN ``Varietà reali e complesse: geometria, topologia e analisi armonica'', Project PRIN 2017 ``Real and Complex Manifolds: Topology, Geometry and holomorphic dynamics''
and by GNSAGA of INdAM}
\begin{document}

\begin{abstract}
We study the space of closed anti-invariant forms on an almost complex manifold, possibly non compact. We construct families of (non integrable) almost complex structures on $\R^4$, such that the space of closed $J$-anti-invariant forms is infinite dimensional, and also $0$- or $1$-dimensional. In the compact case, we construct $6$-dimensional almost complex manifolds with arbitrary large anti-invariant cohomology and a $2$-parameter family of almost complex structures on the Kodaira-Thurston manifold whose anti-invariant cohomology group has maximum dimension.
\end{abstract}
\maketitle
\section{Introduction}
Cohomological properties provides a connection between analytical and topological features of complex manifolds. Indeed for a given complex manifold $(M,J)$, natural complex cohomologies are defined, e.g., the {\em Dolbeault, Bott-Chern} and {\em Aeppli} cohomology groups, given by
$$
H^{\bullet,\bullet}_{\overline{\partial}}(M)=\frac{\Ker\overline{\partial}}{\Imm \overline{\partial}}\,,\,\,
H^{\bullet,\bullet}_{BC}(M)=\frac{\Ker\del\cap\Ker\delbar}{\Imm \del\delbar}\,,\,\,
H^{\bullet,\bullet}_{A}(M)=\frac{\Ker\del\delbar}{\Imm \del+\Imm\delbar}\,.
$$
Furthermore, if $(M,J)$ is a compact complex manifold admitting a K\"ahler metric, that is a $J$-Hermitian metric whose fundamental form is closed, as a consequence of Hodge theory, the complex de Rham cohomology groups decompose as the direct sum of $(p,q)$-Dolbeault groups and strong topological restrictions on $M$ are derived.\newline 
For an almost complex manifold $(M,J)$ the exterior differential $d$ acting on the space of complex 
valued $(p,q)$-forms splits as 
$$
d=\mu+\partial+\overline{\partial}+\overline{\mu},
$$
where $\overline{\partial}$, respectively $\overline{\mu}$ are the $(p,q+1)$ respectively the $(p-1,q+2)$ 
components of $d$. It turns out that the almost complex structure $J$ is integrable if and only if 
$\overline{\mu}=0$. Consequently, in the non integrable case, $\overline{\partial}$ is not a cohomolgical 
operator.\newline 
In \cite{LZ} Li and Zhang, motivated by the study of comparison of tamed and compatible symplectic cones 
on a compact almost complex manifold, introduced the $J$-{\em anti-invariant} and $J$-{\em invariant cohomology groups} as the (real) de Rham $2$-classes represented by $J$-anti-invariant, respectively $J$-invariant forms and the notion of $\hbox{\em \Cpf}$ almost complex structures, namely those ones such that 
the second de Rham cohomology group decomposes as the direct sum of the $J$-anti-invariant and $J$-invariant cohomology groups. In \cite{DLZ1}, Dr\v{a}ghici, Li and Zhang proved that an almost complex 
structure on a compact $4$-dimensional manifold is $\hbox{\rm \Cpf}$.
\newline In \cite{DLZ2} and \cite{DLZ3}, the same authors continue the study of the $J$-anti-invariant cohomology of an almost complex manifold $(M,J)$. {Let $h^-_J$ be the dimension of the real vector space of closed anti-invariant $2$-forms on $(M,J)$. Note that in the case when the manifold is $4$-dimensional every closed anti-invariant form $\alpha$ is $\Delta_{g_J}$-harmonic, where $g_J$ is a Hermitian metric and $\Delta_{g_J}$ denotes the Hodge Lapacian, see section \ref{preliminaries}. Thus in the compact $4$ dimensional case $h^-_J$ is the dimension of the anti-invariant cohomology. The following conjectures appear in \cite{DLZ2}.} 
\vskip.1truecm\noindent
{\bf Conjecture 2.4}. {\em For generic almost complex structures $J$ on a compact $4$-manifold $M$, $h^-_J= 0$.}
\vskip.1truecm\noindent
In the case when $b^+ =1$ this was proved as Theorem 3.1 the same paper. The conjecture in general was established by Tan, Wang, Zhang and Zhu in \cite{twzz}.
\vskip.1truecm\noindent
{\bf Conjecture 2.5}. {\em On a compact $4$-manifold, if $h^-_J\geq 3$, then $J$ is integrable.}
\vskip.1truecm\noindent
By starting with a (compact) K\"ahler surface with holomorphically trivial canonical bundle, Dr\v{a}ghici, Li and Zhang obtain non integrable almost complex structures with $h^-_J=2$. More precisely, for a given (compact) K\"ahler surface $(M,J)$ with holomorphically trivial canonical bundle, they take a closed $2$-form trivializing the canonical bundle. Then, fixing a conformal class of Hermitian metrics compatible with $J$, they consider the Gauduchon metric representing such a conformal class and they associate an almost complex structure $J_{f,s,l}$ depending on three smooth functions satisfying some suitable conditions. Then, generically, $h^-_{J_{f,s,l}}=0$, but cases when $h^-_{J_{f,s,l}}=1$ and $h^-_{J_{f,s,l}}=2$ also occur. 
Therefore, again in \cite{DLZ2}, as an extension of Conjecture 2.5, the authors asked the following natural
\vskip.1truecm\noindent
{\bf Question 3.23}. {\em  
Are there (compact, $4$-dimensional) examples of non-integrable
almost complex structures $J$ with $h^{-}_J\geq 2$ other than the ones arising from
\cite{DLZ2}, Proposition 3.21? In particular, are there any examples with $h^-_J \geq 3$?}

\vskip.2truecm\noindent
%Note that if $(M,J)$ is a $4$-dimensional almost complex manifold and $g_J$ is any Hermitian metric on $(M,J)$, then every $J$-anti-invariant closed form $\alpha$ is $\Delta_{g_J}$-harmonic, where $\Delta_{g_J}$ denotes the Hodge Lapacian. Indeed, $\alpha$ is selfdual and consequently coclosed.
For other results on $\hbox{\rm \Cpf}$ and $J$-anti-invariant closed forms see \cite{AT,ATZ,BZ,HMT1,LU}.\newline

In this note, motivated by Conjecture 2.5 and Question 3.23, we study the anti-invariant cohomology and the space of anti-invariant harmonic forms of an almost complex manifold, possibly non compact.

Starting with the non compact case, we first note that the space of closed anti-invariant forms with respect to the standard integrable complex structure $i$ on $\R^4 \equiv \C^2$ is infinite dimensional: indeed, for every given holomorphic function $h(z_1,z_2)$, the real and imaginary parts of $h(z_1, z_2)dz_1\wedge dz_2$ are closed and anti-invariant.

{As Theorem \ref{infinite-harmonic}, we show the same can also hold in the non integrable case.}

\vskip.1truecm\noindent

{{\bf Theorem.} {\em There exists a (non integrable) almost complex structure on $\R^4$, such that the space of closed $J$-anti-invariant forms is infinite dimensional.}}

\vskip.1truecm\noindent

%We construct a (non integrable) almost complex structure on $\R^4$, such that the space of closed $J$-anti-invariant forms is infinite dimensional (Theorem \ref{infinite-harmonic}). 
{As a consequence, we see that compactness is essential for Conjecture 2.5.}

In contrast we also show the following (see Theorem \ref{limit}, and Lemma \ref{integrability-R4} for the integrability statement).

\vskip.1truecm\noindent

{{\bf Theorem.} {\em There exists a family of almost complex structures $\{J_f\}$ on $\C^2$, parameterized by smooth functions $f: \C^2 \to \R$, with the following properties. \begin{itemize} \item $J_f$ coincides with the standard complex structure $i$ exactly at points where $f=0$;
 \item $J_f$ is integrable if and only if the gradient of $f$ in the $z_2$ direction is $0$;
 \item if $f$ has compact support and $f \not\equiv 0$ then $h^-_{J_f}=1$.
 %\item $h^-_{J_f} = \dim_\R\mathcal{H}^-_{J_f}(\R^4)=1$.
\end{itemize}
}}

\vskip.1truecm\noindent

{In particular, an arbitrarily small, compactly supported, perturbation of a complex structure having an infinite dimensional space of anti-invariant forms may admit only a single such form up to scale. This provides supporting evidence for Conjecture 2.5, showing that typically anti-invariant forms do not persist under nonintegrable perturbations. }
%(Notice that, if $J_0$ is the standard complex structure on $\C^2$, with coordinates $(z_1,z_2)$, then the space $\mathcal{H}^-_{J_0}(\C^2)$ of harmonic anti-invariant $2$-forms has infinite dimension: indeed, for every given holomorphic function $f(z_1,z_2)$, the real and imaginary part of $f(z)dz_1\wedge dz_2$ belong to $\mathcal{H}^-_{J_0}(\C^2)$). %On the other hand, there exist non-integrable almost complex structures such that $\mathcal{H}^-_{J}$ is of infinite respectively finite dimension, as showed in theorems \ref{infinite-harmonic} and \ref{limit}. 

{A similar argument gives the following, see Corollary \ref{dimension0}.}

\vskip.1truecm\noindent

{{\bf Corollary.} {\em There exist almost complex structures on $\C^2$ which agree with $i$ outside of a compact set and have $h^-_J =0$.}}

\vskip.1truecm\noindent

{We note that integrable complex structures on $\C^2$ which agree with $i$ outside of a compact set are biholomorphic to $\C^2$ itself, and so have $h^-_J = \infty$. This follows from Yau, \cite{yau}, Theorem 5, since such complex structures can be extended to give complex structures on $\C P^2$.}

{Given the original motivations  for studying anti-invariant cohomology groups it is natural to ask about compatibility properties for our almost complex structures. We point out in Remark \ref{almostkahler} that the almost complex structures described in both of the above theorems are indeed almost K\"{a}hler, that is, they are compatible with a symplectic form on $\C^2$.}

In the compact case, we construct a $2$-parameter family of (non integrable) almost complex structures on the Kodaira-Thurston manifold, depending on two smooth functions, for which the anti-invariant cohomology 
group has maximum dimension equal to $2$ (see Proposition \ref{kodaira}). This provides an affirmative answer to Question 3.23. In the last section, we give a 
simple construction to obtain $6$-dimensional compact almost complex manifolds with arbitrary large anti-invariant cohomology (see Proposition \ref{large-anti-invariant}). Hence dimension $4$ is also an essential part of Conjecture 2.5.

For almost-complex structure on a 4-manifold which are tamed by a symplectic form, Dr\v{a}ghici, Li and Zhang show in \cite{DLZ1}, Theorem 3.3, that $h^-_J \le b^+ -1$. Thus any counterexamples to Conjecture 2.5 cannot come from tame almost-complex structures on symplectic 4-manifolds with $b^+ \le 3$. Moreover T.-J. Li in \cite{TJL}, Theorem 1.1, shows that symplectic 4-manifolds of Kodaira dimension $0$ all have $b^+ \le 3$. We thank Weiyi Zhang for pointing this out.

\vskip.2truecm\indent
\noindent{\sl Acknowledgments.} The second author would like to thank the Math Department of Notre Dame University for its warm hospitality, and we both thank Tedi Dr\v{a}ghici for valuable remarks including pointing out an error in the formulas justifying Theorem \ref{infinite-harmonic}, and Weiyi Zhang for more comments and insight. We also thank an anonymous referee for helping to clarify the presentation and simplifying the statement of Theorem \ref{limit}.

\section{Anti-Invariant cohomology}\label{preliminaries}
In this Section we will fix some notation and recall the generalities on anti-invariant forms and some 
notion about the cohomology of almost complex manifolds. Let $M$ be a smooth $2n$-dimensional manifold. We
will denote by $J$ a smooth almost complex structure on $M$, that is a smooth $(1,1)$-tensor $J$ field 
satisfying $J^2=-\hbox{\rm id}$. The almost complex structure $J$ is said to be {\em integrable} if its 
Nijenhuis tensor, that is the $(1,2)$-tensor given by
$$
N_J(X,Y)=[JX,JY]-[X,Y]-J[JX,Y]-J[X,JY],
$$
According to Newlander-Nirenberg Theorem, $J$ is integrable if and only if $J$ is induced by a structure of complex manifold on $M$. 
%A smooth map $\phi:(M,J)\to(M',J')$ is said to be {\em pseudo-holomorphic} if 
%$$
%d\phi\circ J=J'\circ d\phi
%$$ 
Let $J$ be a smooth almost-complex structure on a $M$ and denote 
by $\Lambda^r(M)$ the bundle of $r$-forms on $M$; let $\Omega^r(M):=\Gamma(M,\Lambda^r(M))$ be 
the space of smooth global sections of $\Lambda^r(X)$ and let $\Lambda^r(M;\C)=\Lambda^r(M)\otimes\C$. 
Then $J$ acts in a natural way on the space $\Omega^r(M;\C)$ of smooth sections of 
$\Lambda^r(M;\C)$ giving rise to the following bundle decomposition
$$
\Lambda^r(M;\C)=\bigoplus_{p+q=r}\Lambda^{p,q}_J(M).
$$
Accordingly, $\Omega^r(M;\C)$ and $\Omega^r(M)$ decompose respectively as 
$$
\Omega^r(M;\C)=\bigoplus_{p+q=r}\Omega^{p,q}_J(M).
$$
and
$$
\Omega^r(M)=\bigoplus_{p+q=r,\,p\leq q}\Omega^{(p,q),(q,p)}(M)_\R,
$$
where, for $p <q$ 
$$
\Omega^{(p,q),(q,p)}(M)_\R=\{\alpha\in\Omega^{p,q}_J(M)\oplus \Omega^{q,p}_J(M)\,\,\,\vert\,\,\, 
\alpha=\overline{\alpha}\}
$$
and
$$
\Omega^{(p,p)}(M)_\R=\{\beta\in\Omega^{p,p}_J(M)\,\,\,\vert\,\,\, 
\beta=\overline{\beta}\}
$$
In particular for $r=2$, $J$ acts as involution on   
$\Omega^2(M)$ by 
$$J\alpha(X,Y)=\alpha(JX,JY),
$$ 
for every pair of vector fields $X$, $Y$ on $M$. Then we denote as usual by 
$\Lambda^-_J(M)$ (respectively $\Lambda^+_J(M)$) the 
$+1$ (resp. $-1$)-eigenbundle; then the space of corresponding sections $\Omega^-_J(M)$ (respectively 
$\Omega^+_J(M)$)
%$$
%\Omega^-_J(M)=\Omega^{(2,0),(0,2)}(M)_\R
%$$
%and 
%$$
%\Omega^+_J(M)=\Omega^{(1,1)}(M)_\R\newline
%$$
are defined to be the spaces of $J$-{\em anti-invariant}, (respectively 
$J$-{\em invariant}) forms, i.e., 
$$
\Omega_J^\pm(M)=\{\alpha\in \Omega^2(M)\,\,\,\vert\,\,\,J\alpha=\pm \alpha\}
$$  
$$
\Omega^{(2,0),(0,2)}(M)_\R=\Omega^-_J(M),\quad
\Omega^{1,1}(M)_\R=\Omega^+_J(M)
$$
%%
%%
%Denote as usual by $\Lambda^+_J(M)$ (resp. $\Lambda^-_J(M)$ the 
%$+1$ (resp. $-1$)-eigenbundle and by $\Omega^+_J(M)$, (resp. $\Omega^-_J(M)$) the corresponding sections, %namely the spaces of $J$-invariant, (resp. 
%$J$-anti-invariant) forms, i.e., 
%$$
%\Omega_J^\pm(M)=\{\alpha\in \Omega^2(M)\,\,\,\vert\,\,\,J\alpha=\pm \alpha\}
%$$
Let 
$$
\mathcal{Z}^\pm_J(M)=\mathcal{Z}^2(M)\cap \Omega^\pm_J(M)=\{\alpha\in \Omega^\pm_J(M)\,\,\,\vert\,\,\,d\alpha =0\}.
$$
If $\{\varphi^1,\ldots,\varphi^n\}$ is a local coframe of $(1,0)$-forms on $(M,J)$, then $\Lambda^-_J(M)$ is locally spanned by
$$
\{
\hbox{\rm Re}(\varphi^r\wedge\varphi^s), \quad \hbox{\rm Im}(\varphi^r\wedge\varphi^s),\,\, 1\leq r<s \leq n 
\}.
$$
Then, according to the previous decomposition on forms, T.-J. Li and W. Zhang \cite{LZ} defined 
the following cohomology spaces 
$$
H^\pm_J (X)= \left\{ \mathfrak{a} \in H^2_{dR}(X;\R) \,\,\, \vert \, \,\,\exists\,\alpha \in {\mathcal Z}^{\pm}_J\,\,\vert\,\, \mathfrak{a}=[\alpha]\right\} 
$$
and they gave the following (see \cite[Definition 4.12]{LZ})
\begin{definition}
An almost complex structure $J$ on $M$ is said to be 
\begin{enumerate}
\item[$\bullet$] {\em\Cp}\ if
$$
H^{+}_J (M) \cap H^{-}_J (M)=\{0\}\,.
$$
\item[$\bullet$] {\em\Cf}\ if
$$
H^{2}_{dR} (M;\R) =H^{+}_J (M) + H^{-}_J (M).
$$
\item[$\bullet$] {\em\Cpf}\  if
$$
H^{2}_{dR} (M;\R) =H^{+}_J (M) \oplus H^{-}_J (M).
$$
\end{enumerate}
\end{definition}
Given an almost complex manifold $(M,J)$, {we denote by 
$$
h^-_J(M)=\dim_\R \mathcal{Z}^-_J(M).
$$
%Furthermore, we will say that an almost-complex manifold $\left(X,J\right)$ is \emph{\Cpf\ at the $k$-th stage} if the decomposition
%$$ 
%H^k_{dR}(X;\R) = \bigoplus_{p+q=k}H^{(p,q),(q,p)}_J(M)_\R 
%$$
%holds, where 
%$$
%H^{(p,q),(q,p)}_J(X)_\R \;=\; \left\{\left[\alpha\right]\in
%H^{p+q}_{dR}(X;\R) \,\,\,\vert\,\,\,
%\alpha\in\left(A^{p,q}_J(X)\oplus A^{q,p}_J(X)\right)\,\cap\,A^{p+q}(X)\right\}.
%$$
%Similar definitions for currents can be given, introducing the notion of \emph{\pf} almost-complex 
%structure (see \cite[\S2.2.2]{LZ} for further details and results).\newline 
%By definition, an {\em almost K\"ahler structure} on a manifold $X$ is a triple $(J,g,\omega)$ where $J$ %is an almost complex, 
%$g$ is a $J$-Hermitian metric and 
%the $2$-form $\omega (\cdot,\cdot):=g(\cdot,J\cdot)$ is closed. In such a case, $J$ is said to be $\omega%$-{\em compatible}.\newline

For a given Hermitian metric $g_J$ on the $2n$-dimensional almost complex manifold $(M,J)$, we will denote by $\mathcal{H}_J^-(M)$ the space of $J$-anti-invariant harmonic $2$-forms, that is
$$
\mathcal{H}_J^-(M):=\{\alpha\in\Omega^-(M)\,\,\,\vert\,\,\,\Delta_{g_J}\alpha=0\},
$$
where $\Delta_{g_J}$ denotes the Hodge Laplacian.

{Following \cite{DLZ2}, \cite[Prop.2.4]{HMT2}, once a $J$-Hermitian metric $g_J$ is fixed, the space 
$\mathcal{Z}^-_J(M)$ is contained in the kernel of a second order
elliptic differential operator $\mathbb{E}$, that is $\mathcal{Z}^-_J(M) \subset \hbox{\rm Ker}\,\mathbb{E}$. Explicitly, 
$$
E\alpha= \Delta_{g_J}\alpha+\frac{1}{(n-2)!}d((\alpha\wedge d(\omega^{n-2}))),
$$
where $\omega$ is the fundamental form of $g_J$. Hence, if $M$ is a compact $2n$-dimensional almost complex manifold, then 
$\mathcal{Z}^-_J(M)$ has finite dimension. Also, in view of \cite{A}, assuming $M$ is connected, if $\alpha$ is any closed anti-invariant form vanishing to infinite order at some point $p$, then $\alpha=0$. }

{In the case when $2n=4$, then any $J$-anti-invariant closed form $\alpha$ on $(M,J)$ satisfies $\Delta_{g_J}\alpha = \mathbb{E}\alpha= 0$ and so $\mathcal{Z}^-_J(M) \subset \mathcal{H}_J^-(M)$.  Thus if $M$ is compact the natural map $\mathcal{Z}^-_J(M) \hookrightarrow H_J^-(M)$ is an isomorphism. This also holds for compact $M$ in higher dimensions provided that $J$ is compatible with a symplectic form, that is, $(M,J)$ is an almost K\"ahler manifold, see for example, \cite{DLZ2} or \cite[Proposition 2.2, Corollary 2.3]{HMT2}.}
%Therefore, in view of \cite{A} on a connected $2n$-dimensional almost complex manifold $M$, if $\alpha$ is any closed anti-invariant form vanishing to infinite order at some point $p$, then $\alpha=0$. 

{Finally, again in dimension $2n=4$, we can check that in fact $\mathcal{Z}^-_{J}(M)\subset \mathcal{H}^+_{g_J} \subset \mathcal{H}_J^-(M)$ where 
$\mathcal{H}^+_{g_J}$ is the space of self-dual harmonic forms. So in the compact case we have $h^-_J(M)\leq b^+(M)$. }

\section{Closed $J$-anti-invariant forms and an integrability condition}
Let $J$ be an almost complex structure on a $4$-dimensional manifold. Let $\omega\neq 0$ be a closed $J$-anti-invariant form on $M$. Then, according to \cite[Lemma 2.6]{DLZ1} (see also \cite[Prop. 2.6]{HMT2}) the zero set 
$\omega^{-1}(0)$ of $\omega$ has empty interior, so that $M\setminus\omega^{-1}(0)$ is open and dense. Since $M\setminus\omega^{-1}(0)$ coincides with the subset of $M$ where $\omega$ is non degenerate (see 
\cite[Lemma 2.6]{DLZ1} or \cite[Lemma 1.1]{HMT2}), we have the following
\begin{lemma}\label{symplectic}
Let $(M,J)$ be a $4$-dimensional almost complex manifold and $0\neq \omega\in\mathcal{Z}^-_J$. Then 
$\omega$ is a symplectic form on the open dense set $M\setminus\omega^{-1}(0)$.
\end{lemma}
Let $J_0$ be the standard complex structure on the vector space $\C^n\simeq\R^{2n}$ induced by the multiplication by $i$, that is,
$$
J_0(z_1,\ldots,z_n)=(e^{i\frac{\pi}{2}}z_1,\ldots,e^{i\frac{\pi}{2}}z_n).
$$
Then, for every given real number $r$, define $J_0^r\in\hbox{\rm End}(\C^n)$, by setting
$$
J^r_0(z_1,\ldots,z_n)=(e^{i\frac{\pi}{2}r}z_1,\ldots,e^{i\frac{\pi}{2}r}z_n).
$$
Let now $J$ be any almost complex structure on the manifold $\C^n\simeq\R^{2n}$; then there exists 
$A:\R^{2n}\to \hbox{\rm GL}(2n,\R)$ such that $J$ is conjugated 
to the standard complex structure $J_0$, i.e., 
$$J_x=A(x)J_0A^{-1}(x).
$$
For $r=r(x)\in\R$, define 
$$
J_x^r:=A(x)J_0^rA^{-1}(x).
$$

Let $(M,J)$ be a $2n$-dimensional almost complex manifold and let $\omega\in\Omega^-_J(M)$.
{Let $\mathcal{U}$ be a coordinate neighbourhood. We can find $A(x)$ for $x \in  \mathcal{U}$ conjugating $J_x$ to $J_0$. Given a smooth function $r : M \to \R$ equal to $0$ outside of $\mathcal{U}$ we can
define a bilinear form $\theta^r$ on $M$ which agrees with $\omega$ outside of $\mathcal{U}$ by setting, at any given $x\in 
\mathcal{U}$,}
\begin{equation}\label{thetadef}
 \theta^r_x(v,w)=\omega_x(v,J^{r(x)}_xw),
\end{equation}
for every pair of tangent vectors $v$, $w$.
\begin{lemma}\label{theta}
The form $\theta^r$ is skew-symmetric and $J$-anti-invariant, that is, $\theta^r \in \Omega^-_J(M)$.
\end{lemma}
\begin{proof}
For any given pair of tangent vectors $v$, $w$ at $x$,
$$
\begin{array}{l}
J_x\theta^r_x(v,w)=\theta^r_x(J_xv,J_xw)=\omega_x(J_xv,J^{r+1}_xw)=-\omega_x(v,J^{r}_xw)=-\theta^r_x(v,w), 
\end{array}
$$
that is $J\theta^r=-\theta^r$.

{Note that when $r=0$ we have $\theta^0 = \omega$ is skew.
To check $\theta^r$ is skew for all $r$, we fix $x$ (and so can think of $r$ as a real number) and working in $T_x M$ can choose a basis such that we can identify $J$ with the standard complex structure $J_0$ on $\C^n$. Then 
$$\frac{d}{dr}J^r = \frac{\pi}{2} J^{r+1}.
$$
Hence 
$$\frac{d}{dr} \theta^r(v,w) =\frac{d}{dr} \omega(v,J^r w)=\frac{\pi}{2}\omega(v,J^{r+1} w)=
\frac{\pi}{2} \theta(v,Jw).
$$}

For the same fixed $v$, $w$, we define a function $$f(r) = ( \theta^r(v,w) + \theta^r(w,v) )^2 + ( \theta^r(v, Jw) + \theta^r(Jw,v) )^2.$$

Then $$\frac{df}{dr} = \frac{\pi}{2}( \theta^r(v,w) + \theta^r(w,v) ) ( \theta^r(v,Jw) + \theta^r(w,Jv) ) $$ $$ + \frac{\pi}{2}( \theta^r(v, Jw) + \theta^r(Jw,v) )( - \theta(v, w) + \theta(Jw,Jv)) =0$$
using the fact that $J\theta^r = -\theta^r$. Hence $f'(r) = 0$ and since $f(0)=0$ we see that $f(r)=0$ for all $r$ and $\theta^r$ is skew for all $r$.

\end{proof}
The last Lemma allows to produce anti-invariant forms starting from an anti-invariant one. 
For the sake of completeness we recall the proof of an integrability result in the $4$-dimensional case obtained by 
Dr\v{a}ghici, Li and Zhang (see \cite[Lemma 2.12]{DLZ1}).
{\begin{prop}
Let $(M,J)$ be a $4$-dimensional almost complex manifold. Let $0 \neq \omega\in\mathcal{Z}^-_J(M)$. %Assume that, around at any $x\in M\setminus \omega^{-1}(0)$, the local form 
If the form $\theta_x(\cdot,\cdot)=\omega_x(\cdot,J_x\cdot)$ is closed, then $J$ is integrable.  
\end{prop}}
\begin{proof}
{It suffices to check the Nijenhuis tensor $N_J=0$, at any point of the dense subset $M\setminus \omega^{-1}(0)$. This implies $N_J=0$ on the whole $M$ and $J$ is integrable.}

By Lemma \ref{symplectic} the $2$-form $\omega$ is a symplectic structure on $M\setminus \omega^{-1}(0)$. Let 
$x\in M\setminus \omega^{-1}(0)$ and $\mathcal{U}$ be a coordinate neighbourhood of $x$ contained in 
$M\setminus \omega^{-1}(0)$. Define a local complex $2$-form on $(M,J)$ by setting, for every $x\in\mathcal{U}$, 
$$
\Psi_x=\omega_x-i\theta_x.
$$
We show that $\Omega$ is of type $(2,0)$.
Indeed, for every given $v$, $w$, 
$$
\begin{array}{ll}
\Psi_x(v-iJv,w+iJw)&=(\omega_x-i\theta_x)(v-iJv,w+iJw)\\[5pt]
&= \omega_x(v,w)+\omega_x(Jv,Jw)-i\big(\theta_x(v,w)+\theta_x(Jv,Jw)\big)\\[5pt]
&+ i\big(\omega_x(v,Jw)-\omega_x(Jv,w)-i(\theta_x(v,Jw)-\theta_x(Jv,w))\big)\\[5pt]
&= 0,
\end{array}
$$
since $\omega$ and $\theta$ are $J$-anti-invariant. Therefore, $\Psi$ vanishes on any pair of 
complex vectors of type $(1,0)$, $(0,1)$, respectively, that is 
$$
\Psi\in\Omega^{2,0}_J(\mathcal{U})\oplus \Omega^{0,2}_J(\mathcal{U}).
$$
Similarly, 
$$
\begin{array}{ll}
\Psi_x(v+iJv,w+iJw)&=(\omega_x-i\theta_x)(v+iJv,w+iJw)\\[5pt]
&= \omega_x(v,w)-\omega_x(Jv,Jw)-i\big(\theta_x(v,w)-\theta_x(Jv,Jw)\big)\\[5pt]
&+ i\big(\omega_x(v,Jw)+\omega_x(Jv,w)-i\big(\theta_x(v,Jw)+\theta_x(Jv,w)\big)\big)\\[5pt]
&= 2\big(\omega_x(v,w)-i\theta_x(v,w)\big)+2i\big(\omega_x(v,Jw)-i\theta_x(v,Jw)\big)\\[5pt]
&= 2\big(\omega_x(v,w)-i\theta_x(v,w)\big)+2i\big(\theta_x(v,w)+i\omega_x(v,w)\big)\\[5pt]
&= 0.
\end{array}
$$
Therefore, $\Psi\in\Omega^{2,0}_J(\mathcal{U})$ is nowhere vanishing and closed. Let $\alpha$ be any local 
complex $(1,0)$-form. Then, by type reason, $\alpha\wedge\Psi=0$. Hence, at $x$,
$$
0=d(\alpha\wedge\Psi)=d\alpha\wedge\Psi=(d\alpha)^{0,2}\wedge\Psi,
$$
which implies that the $(0,2)$-part $(d\alpha)^{0,2}$ of $d\alpha$ vanishes and $N_J(x)=0$. \newline 

\end{proof}

Let $(x_1,x_2,y_1,y_2)$ be natural coordinates on $\R^4$ and $f=f(x_1,x_2,y_1,y_2)$ be a smooth $\R$-valued function on $\R^4$. Define $J_f\in\hbox{\rm End}(T\R^4)$ by setting 
\begin{equation}\label{almost-complex-structure}
J_f\frac{\partial}{\partial x_1}=f\frac{\partial}{\partial x_2}+\frac{\partial}{\partial y_1},\,
J_f\frac{\partial}{\partial x_2}=\frac{\partial}{\partial y_2},\,
J_f\frac{\partial}{\partial y_1}=-\frac{\partial}{\partial x_1}-f\frac{\partial}{\partial y_2},\,
J_f\frac{\partial}{\partial y_2}=-\frac{\partial}{\partial x_2}
\end{equation}
and extend it $\mathcal{C}^\infty(\R^4)$-linearly. Then $J_f$ gives rise to an almost complex strcuture on $\R^4$. 
\begin{lemma}\label{integrability-R4}
The almost complex structure $J=J_f$ is integrable if and only if 
$$
f_{x_2}=0,\quad f_{y_2}=0.
$$
\end{lemma}
\begin{proof} It is enough to show that $N_J(\frac{\partial}{\partial x_1},\frac{\partial}{\partial x_2})=0$ if and only if 
$$
f_{x_2}=0,\quad f_{y_2}=0.
$$
We easily compute
$$
\begin{array}{ll}
N_J(\frac{\partial}{\partial x_1},\frac{\partial}{\partial x_2})&= 
[J\frac{\partial}{\partial x_1},J\frac{\partial}{\partial x_2}]-[\frac{\partial}{\partial x_1},\frac{\partial}{\partial x_2}]-J[J\frac{\partial}{\partial x_1},\frac{\partial}{\partial x_2}]
-J[\frac{\partial}{\partial x_1},J\frac{\partial}{\partial x_2}]\\[5pt]
&=[f\frac{\partial}{\partial x_2}+\frac{\partial}{\partial y_1},\frac{\partial}{\partial y_2}]
-J[f\frac{\partial}{\partial x_2}+\frac{\partial}{\partial y_1},\frac{\partial}{\partial x_2}]
-J[\frac{\partial}{\partial x_1},\frac{\partial}{\partial y_2}]\\[5pt]
&=-f_{y_2}\frac{\partial}{\partial x_2}+f_{x_2}\frac{\partial}{\partial y_2}
\end{array}
$$
Lemma is proved.
\end{proof}
According to the definition of $J_f$, the induced almost complex structure $J_f$ on $T^*\R^4$ is given by
\begin{equation}\label{dual-J}
J_fdx_1=-dy_1,\quad J_fdx_2=fdx_1-dy_2, \quad J_fdy_1=dx_1, \quad J_fdy_2=-fdy_1+dx_2.
\end{equation}
Consequently, setting
$$
\varphi^1=dx_1+idy_1,\quad \varphi^2=dx_2+i(-fdx_1+dy_2),
$$
then $\{\varphi^1,\varphi^2\}$ is a complex $(1,0)$-coframe on the almost complex manifold $(\R^4,J_f)$, so that 
$$
\beta=\hbox{\rm Re}(\varphi^1\wedge\varphi^2), \quad \gamma=\hbox{\rm Im}(\varphi^1\wedge\varphi^2),
$$
is a global frame of $\Lambda^-_{J_f}(\R^4)$. Explicitly,
\begin{equation}
\beta = dx_1\wedge dx_2-fdx_1\wedge dy_1- dy_1\wedge dy_2,\quad
\gamma = dx_1\wedge dy_2-dx_2\wedge dy_1.
\end{equation}
\begin{lemma}\label{first-lemma}
Let $\alpha$ be an arbitrary smooth section of $\Lambda^-_{J_f}(\R^4)$. Set 
$$
\alpha =a\beta+b\gamma
$$ 
for $a$, $b$ smooth $\R$-valued functions on $\R^4$. Then $d\alpha=0$ if and only if the following condition holds
\begin{equation}\label{PDE-system}
 \left\{
\begin{array}{lll}
a_{y_1}-b_{x_1}+(fa)_{x_2}&=& 0 \\[3pt]
a_{x_1}+b_{y_1}+(fa)_{y_2}&=& 0 \\[3pt]
a_{y_2}-b_{x_2}&=& 0 \\[3pt]
a_{x_2}+b_{y_2}&=& 0 
\end{array}
\right.
\end{equation}
\end{lemma}
\begin{proof}
Expanding $d\alpha$ we get:
$$
\begin{array}{ll}
d\alpha &= da\wedge\beta - a df\wedge dx_1\wedge dy_1 +db\wedge \gamma\\[5pt]
&= (a_{x_1}dx_1+a_{x_2}dx_2+a_{y_1}dy_1+a_{y_2}dy_2)\wedge (dx_1\wedge dx_2-fdx_1\wedge dy_1- dy_1\wedge dy_2)\\[5pt]
& -a(f_{x_1}dx_1+f_{x_2}dx_2+f_{y_1}dy_1+f_{y_2}dy_2)\wedge dx_1\wedge dy_1+\\[5pt]
& + (b_{x_1}dx_1+b_{x_2}dx_2+b_{y_1}dy_1+b_{y_2}dy_2)\wedge (dx_1\wedge dy_2-dx_2\wedge dy_1)\\[5pt]
&= -a_{x_1} dx_1\wedge dy_1\wedge dy_2 + a_{x_2}f dx_1\wedge dx_2\wedge dy_1 -
a_{x_2} dx_2\wedge dy_1\wedge dy_2+\\[5pt]
& +a_{y_1} dx_1\wedge dx_2\wedge dy_1 + a_{y_2} dx_1\wedge dx_2\wedge dy_2 -
a_{y_2}f dx_1\wedge dy_1\wedge dy_2\\[5pt]
& + af_{x_2} dx_1\wedge dx_2\wedge dy_1-af_{y_2} dx_1\wedge dy_1\wedge dy_2 - b_{x_1} dx_1\wedge dx_2\wedge dy_1+ \\[5pt] 
& - b_{x_2} dx_1\wedge dx_2\wedge dy_2  - b_{y_1} dx_1\wedge dy_1\wedge dy_2 - b_{y_2} dx_2\wedge dy_1\wedge dy_2\\[5pt]
&=(a_{y_1}-b_{x_1}+ (af)_{x_2})\,dx_1\wedge dx_2\wedge dy_1+(a_{y_2}-b_{x_2})\,dx_1\wedge dx_2\wedge dy_2+\\[5pt]
&- (a_{x_1}+b_{y_1}+ (af)_{y_2})\,dx_1\wedge dy_1\wedge dy_2-
(a_{x_2}+b_{y_2})\,dx_2\wedge dy_1\wedge dy_2.
\end{array}
$$
Therefore, $d\alpha=0$ if and only if \eqref{PDE-system} holds.
\end{proof}
\begin{rem} Set $z_1=x_1+iy_1,$ $z_2=x_2+iy_2$ and
$$
\begin{array}{ll}
 \partial_{z_1}=\frac12(\partial_{x_1}-i\partial_{y_1}),&
\partial_{z_2}=\frac12(\partial_{x_2}-i\partial_{y_2})\\[5pt]
\partial_{\overline{z}_1}=\frac12(\partial_{x_1}+i\partial_{y_1}),&
\partial_{\overline{z}_2}=\frac12(\partial_{x_2}+i\partial_{y_2})
\end{array}
$$
Then a pair of real valued functions $(a,b)$ on $\R^4$ is a solution of \eqref{PDE-system} 
if and only if the complex valued function $w=a-ib$ solves the following

\begin{equation}\label{PDE-system-complex-general}
 \left\{
\begin{array}{l}
\partial_{\overline{z}_1}w +\frac{i}{2}\partial_{z_2}(f(w+\overline{w}))= 0 \\[5pt]
\partial_{\overline{z}_2}w = 0 
\end{array}
\right.
\end{equation}

The system above is a perturbed Cauchy-Riemann PDEs system.
Furthermore, it is immediate to note that, condition \eqref{PDE-system} of Lemma \ref{first-lemma} can be rewritten as
$$
db=(a_{y_1}+(af)_{x_2})dx_1+ a_{y_2}dx_2- (a_{x_1}+(af)_{y_2})dy_1-a_{x_2}dy_2.
$$
Therefore, given $a$, there exists a $b$ such that $\alpha = a \beta + b \gamma$ is a closed $J$-anti-invariant form on $(\R^4,J)$ if and only if the differential form 
$$
(a_{y_1}+(af)_{x_2})dx_1+ a_{y_2}dx_2- (a_{x_1}+(af)_{y_2})dy_1-a_{x_2}dy_2
$$
is closed. The latter condition is equivalent to the following PDEs system:
\begin{equation}\label{system2}
\left\{
\begin{array}{lll}
 a_{x_1y_2}-a_{x_2y_1}-(af)_{x_2x_2}&=& 0 \\[3pt]
 a_{x_1y_2}-a_{x_2y_1}+(af)_{y_2y_2}&=& 0\\[3pt]
 a_{x_1x_1}+a_{y_1y_1}+(af)_{x_2y_1}+(af)_{x_1y_2}&=& 0 \\[3pt]
 a_{x_1x_2}+a_{y_1y_2}+(af)_{x_2y_2}&=& 0 \\[3pt]
 a_{x_2x_2}+a_{y_2y_2}&=& 0 
\end{array}
\right.
\end{equation}
\end{rem}

%Let $g_J$ be a $J$-Hermitian metric on $\R^4$. Denote by $\mathcal{H}^-_{J}(\R^4)$ the space of harmonic $J$-anti-invariant forms on $\R^4$. 
We are ready to state and prove the following 
\begin{theorem}\label{infinite-harmonic}
Let $f(x_1,x_2,y_1,y_2)=x_2$, $J=J_{x_2}$ be defined as in \eqref{almost-complex-structure} and $g_J$ be a $J$-Hermitian metric on $\R^4$. Let
$$
\beta = dx_1\wedge dx_2-fdx_1\wedge dy_1- dy_1\wedge dy_2,\quad
\gamma = dx_1\wedge dy_2-dx_2\wedge dy_1.
$$
Then
\begin{enumerate}
 \item[(I)] $J$ is a non-integrable almost complex structure on $\R^4$.\\[3pt]
 \item[(II)] For every given pair $(s,t)\in\R^2$, such that
 $$
 s^2+t^2+t=0,
 $$
 the form
 $$
 \alpha_{s,t}=te^{sx_1+ty_1}\beta - se^{sx_1+ty_1}\gamma
 $$
 
\end{enumerate}
is a $J$-anti-invariant and closed. Therefore, $\mathcal{H}^-_{J}(\R^4)$ has infinite dimension. 
\end{theorem}
\begin{proof}
(I) In view of Lemma \ref{integrability-R4}, $J$ is integrable if an only if $f_{x_2}=f_{y_2}=0$. By assumption, we get $f_{x_2}=1$. Therefore $J$ is not integrable.\vskip.2truecm\noindent
(II) Set $z_1=x_1+iy_1,$ $z_2=x_2+iy_2$. Then, for $f = x_2$, the complex PDEs system \eqref{PDE-system-complex-general} becomes
%\eqref{PDE-system}
%We rewrite the PDEs system \eqref{PDE-system} in complex notation. \newline
%Set $z_1=x_1+iy_1,$ $z_2=x_2+iy_2$,
%$$
%\begin{array}{ll}
% \partial_{z_1}=\frac12(\partial_{x_1}-i\partial_{y_1}),&
%\partial_{z_2}=\frac12(\partial_{x_2}-i\partial_{y_2})\\[5pt]
%\partial_{\overline{z}_1}=\frac12(\partial_{x_1}+i\partial_{y_1}),&
%\partial_{\overline{z}_2}=\frac12(\partial_{x_2}+i\partial_{y_2})
%\end{array}
%$$
%A pair of real valued functions $(a,b)$ on $\R^4$ is a solution of \eqref{PDE-system} with $f = x_2$ if and only if the complex valued function $w=a-ib$ solves the following 
%$$
% \left\{
%\begin{array}{l}
%\partial_{\overline{z}_1}w -\frac{i}{4}(z_2+\overline{z}_2)(\partial_{\overline{z}_2}-\partial_{z_2})w
%+\frac{i}{4}(w+\overline{w})= 0 \\[5pt]
%\partial_{\overline{z}_2}w = 0 
%\end{array}
%\right.
%$$
%that is,
\begin{equation}\label{PDE-system-complex}
 \left\{
\begin{array}{l}
\partial_{\overline{z}_1}w +\frac{i}{4}(z_2+\overline{z}_2)\partial_{z_2}w
+\frac{i}{4}(w+\overline{w})= 0 \\[5pt]
\partial_{\overline{z}_2}w = 0 
\end{array}
\right.
\end{equation}
%The system above is a perturbed Cauchy-Riemann PDEs system. %Indeed, for $f=0$, we have that $J$ is the standard complex structure in $\C^2$ and \eqref{PDE-system-complex} are the usual Cauchy-Riemann conditions. 

A straightforward computation shows that, given any pair of real numbers $(s,t)$ satisfying 
$$
s^2+t^2+t=0,
$$
the complex function
$$
w = te^{s\frac{z_1-\bar{z}_1}{2}+t\frac{z_1-\bar{z}_1}{2i}}+ise^{s\frac{z_1-\bar{z}_1}{2}+t\frac{z_1-\bar{z}_1}{2i}}
$$
solves \eqref{PDE-system-complex}. Take 
$$
s_n=\frac{\sqrt{n-1}}{n},\qquad t_n=-\frac1n,
$$
Then, for such a choice, $s_n^2+t_n^2+t_n=0$. In view of the computations above, for any given integer $n\geq 1$, the $J$-anti-invariant forms
$$
\alpha_n:=t_ne^{s_nx_1+t_ny_1}\beta - s_ne^{s_nx_1+t_ny_1}\gamma
$$
are closed, and consequently $g_J$-harmonic. Therefore, $\{\alpha_n\}_{n\geq 1}$ is a sequence of harmonic forms on $(\R^4,J,g_J)$ and it is immediate to check that, for any given positive integer $m$, the forms $\{\alpha_1,\ldots, \alpha_m\}$ are linearly independent. This ends the proof.
\end{proof}

{Next we demonstrate the contrasting behavior when our almost complex structure is defined using functions with compact support.}

\begin{theorem}\label{limit}
{Let $f$ have compact support and the almost complex structures $J_f$ on $\C^2$ be defined by \eqref{almost-complex-structure}.}

{Then if $f$ is non-zero we have $h^-_{J_f}=1$.}
%\begin{itemize}
 %\item $J_f\to J_0=i$ on $\C^2$;
 %\item $J_f=i$ outside of $\mathbb{B}(1)$;
 %\item $\dim_\R\mathcal{H}^-_{J_f}(\R^4)=1$.
%\end{itemize}
\end{theorem}

{Note that since $f$ has compact support neither $f_{x_2}$ nor $f_{y_2}$ can vanish identically and so by Lemma \ref{integrability-R4} we see that $J_f$ is nonintegrable. As mentioned in the introduction, Yau's solution to the Calabi conjecture actually implies that no integrable complex structures $J$ can be standard outside of a compact set and satisfy $h^-_{J}=1$.}

\begin{proof}
%We choose an $f$ which vanishes outside of the ball $\mathbb{B}(1)$ but is not identically $0$. Then  
We determine the anti-holomorphic forms by finding solutions to the system \eqref{system2}.

First note that the first two equations in \eqref{system2} imply that $af$ is a harmonic function of $x_2, y_2$, which is identically $0$ outside of a compact set (since $f$ is). Hence $af$ is identically $0$ everywhere.

Fix $x_1, y_1$, say $x_1=s, y_1=t$, so that $f$ does not vanish identically on the corresponding $x_2, y_2$ plane. Working in this plane, as $af$ is identically $0$ it follows that $a$ is identically $0$ on the open set where $f$ is nonzero. But the final equation in \eqref{system2} says that $a$ is also harmonic in $x_2, y_2$, hence $a$ vanishes identically on the whole plane, and similarly on all nearby $x_2, y_2$ planes.

Next we look at $x_1, y_1$ planes. As $af=0$ the third equation in \eqref{system2} says that $a$ is harmonic. But as we know that $a$ is $0$ close to $(s,t)$ we can conclude that $a=0$ everywhere.

Therefore the only closed anti-invariant forms $a \beta + b \gamma$ are of the form $a=0$ and $b$ constant, showing that $h^-_{J_f}=1$ as required.
\end{proof}

Similar almost complex structures give the following corollary.

\begin{cor}\label{dimension0} {There exist almost complex structures on $\C^2$ which agree with $i$ outside of a compact set and have $h^-_J =0$.}
\end{cor}

\begin{proof} {The proof of Theorem \ref{limit} implies that if $J = J_f$ on some region, say $\{ |z_1 -3| < 1\}$ and $f$ is not identically $0$ on the planes $\{z_1 =c\}$ when $|c-3|<1$ then any closed anti-invariant form on $\{ |z_1 -3| < 1\}$ is a multiple of $\gamma$. We fix such an $f$ with support in a ball $B_2(3,0)$ about $(3,0)$ of radius $2$.}

{Consider the mapping $T:\C^2 \to \C^2$, $(z_1, z_2) \mapsto (z_2, -iz_1)$, which takes $\{|z_2 - 3| <1\}$ to $\{ |z_1 -3| < 1\}$. Then $\rho = T^* \gamma = dx_1 \wedge dx_2 - dy_1 \wedge dy_2$ and $J' = T^* J_f$ coincides with $i$ outside of a ball about $B_2(0,3)$. Also, any closed $J'$-anti-invariant form on $\{|z_2 - 3| <1\}$ is a multiple of $\rho$ on $\{|z_2 - 3| <1\}$.}

{Now, both $J$ and $J'$ agree with $i$ outside of the two balls, and so we can find an almost complex structure $J''$ agreeing with $J$ on $B_2(3,0)$ and $J'$ on $B_2(0,3)$ and $i$ away from the two balls. Any corresponding $J''$ anti-invariant form is a multiple of both $\gamma$ and $\rho$ on $\{ |z_1 -3| < 1, |z_2 - 3| <1\}$ and so is equal to $0$ on this region. Hence by unique continuation, see section \ref{preliminaries}, the form must be identically $0$ everywhere.}

\end{proof}

{We conclude this section with a remark about the compatibility of our almost complex structures with symplectic forms.}

\begin{rem}\label{almostkahler}
The almost complex structures referred to in Theorems \ref{infinite-harmonic} and \ref{limit} are almost K\"{a}hler, that is, they are compatible with symplectic forms on $\C^2$. In the case when $f = f(x_1, x_2)$ we can check directly that $J_f$ is compatible with the symplectic form $$\omega_f = dx_1 \wedge dy_1 + dx_2 \wedge dy_2 + f dx_1 \wedge dx_2.$$

In the case when $f$ has compact support the almost complex structure $J_f$ is tamed by $$\omega_K = K dx_1 \wedge dy_1 + dx_2 \wedge dy_2$$ for a sufficiently large constant $K$. This means that $\omega_K (v, J_f v) \ge 0$ with equality only if $v=0$. It then follows from Gromov's theory of pseudoholomorphic curves, \cite{gr}, see also \cite{taubes} for this application, that $J_f$ is in fact compatible with a symplectic form $\omega_c$.

{Standard methods in symplectic geometry, see \cite{moser}, can be used to show that $\omega_f$ and $\omega_c$ are diffeomorphic to the standard symplectic form $\omega_0 = dx_1 \wedge dy_1 + dx_2 \wedge dy_2$, and in fact the diffeomorphisms can be chosen smoothly with $f$. Hence in both theorems we may assume without loss of generality that our almost complex structures are all compatible with $\omega_0$.}

\end{rem}

\section{Families of non-integrable almost complex structures with $h^-_J=2$ on the Kodaira-Thurston manifold} We will recall the construction of the {\em Kodaira-Thurston manifold}. 
Let $\R^4$ be the Euclidean space with coordinate $(x_1,\ldots, x_4)$ endowed with 
the following product $*$: given any $a=(x_1,\ldots,x_4),y=(y_1\ldots, y_4)\in\R^4$, define
$$
x*y=(x_1+y_1,x_2+y_2, x_3+x_1y_2+y_3,x_4+y_4).
$$
Then $(\R^4,*)$ is a nilpotent Lie group and 
$$
\Gamma=\{(\gamma_1,\ldots,\gamma_4)\in\R^4\,\,\,\vert\,\,\, \gamma_j\in\Z, j=1,\ldots,4\}
$$
is a uniform discrete subgroup of $(\R^4,*)$, so that 
$M=\Gamma\backslash\R^4$ is a $4$-dimensional compact manifold. Setting, 
$$
E^1=dx_1,\quad E^2=dx_2,\quad E^3=dx_3-x_1dx_2,\quad E^4=dx_4,\,
$$
then it is immediate to check that $
E^1,E^2,E^3,E^4$ are $\Gamma$-invariant $1$-forms on $\R^4$, and, consequently, they give rise to a 
gobal coframe on $M$. Then the following structure equations hold
$$
dE^1=0,\qquad dE^2=0,\qquad dE^3=-E^1\wedge E^2,\qquad dE^4=0.
$$
Denoting by $\{E_1,\ldots,E_4\}$ the dual global frame on $M$, then
$$
[E_1,E_2]=E_3,
$$
the other brackets vanishing. Let $\lambda=\lambda(x_4)$, $\mu=\mu(x_4)$ be non constant $\R$-valued 
smooth $\Z$-periodic functions. Define an almost complex structure $J=J_{\lambda,\mu}$ on $M$ by setting
\begin{equation}\label{almost-complex-kodaira}
 JE_1=e^{\lambda(x_4)}E_2,\, JE_2=-e^{-\lambda(x_4)}E_1, JE_3=e^{\mu(x_4)}E_4,
 \, JE_4=-e^{-\mu(x_4)}E_3.
\end{equation}
\begin{lemma}
The almost complex structure $J$ is non integrable. 
\end{lemma}
\begin{proof}
We compute
$$
\begin{array}{ll}
N_J(E_1,E_3)&= 
[JE_1,JE_3]-[E_1,E_3]-J[JE_1,E_3]
-J[E_1,JE_3]\\[5pt]
&=[e^{\lambda(x_4)}E_2,e^{\mu(x_4)}E_4]
-J[e^{\lambda(x_4)}E_2,E_3]
-J[E_1,e^{\mu(x_4)}E_4]\\[5pt]
&=-E_4(e^{\lambda(x_4)})E_2=-e^{\lambda(x_4)}\lambda'(x_4)E_2\neq 0
\end{array}
$$
\end{proof}
\begin{prop}\label{kodaira}
 Let $J=J_{\lambda,\mu}$ be the family of the (non invariant) almost complex structures on the Kodaira-Thurston manifold 
 defined as in \eqref{almost-complex-kodaira}. Then $h_J^-(M)=2$.
\end{prop}
\begin{proof}
By the definition of $J$, the following 
$$
\psi^1=E^1+ie^{-\lambda(x_4)}E^2,\quad \psi^2=E^3+ie^{-\mu(x_4)}E^4
$$
is a global $(1,0)$-coframe on $(M,J)$. Then
$$
\theta^1=E^1\wedge E^3- e^{-(\lambda(x_4)+\mu(x_4))} E^2\wedge E^4,\quad
\theta^2=e^{-\mu(x_4)} E^1\wedge E^4+ e^{-\lambda(x_4)} E^2\wedge E^3
$$
globally span $\Lambda^-_J(M)$. We immediately obtain
$$
d\theta^1=0,\quad d(e^{\lambda(x_4)}\theta^2) =0, 
$$
that is $\theta^1$, $e^{\lambda(x_4)}\theta^2$ are closed $J$-anti-invariant forms, hence harmonic, which 
span $\Lambda^-_J(M)$. Since $b^+(M)=2$ and $h^-_J(M)\leq b^+(M)$ for every compact almost complex manifold, we conclude that $h^-_J(M)=2$ and 
$$
H^-_{J}(M)\simeq \hbox{\rm Span}_\R\langle\theta^1,e^{\lambda(x_4)}\theta^2 \rangle.
$$
\end{proof}
\begin{rem}
It should be noted that the two-parameter family of almost complex structures on the Kodaira surface as in Proposition \ref{kodaira} cannot be metric related to an integrable almost complex structure, as, on the contrary, in view of \cite[Proposition 3.20]{DLZ2}, such almost complex structures have $h^{-}_{J_{\lambda,\mu}}\leq 1$.
\end{rem}

\section{$6$-dimensional compact almost complex manifolds with arbitrarily large anti-invariant cohomology}
In this Section we provide simple examples of compact $6$-dimensional manifolds endowed with a non integrable almost complex structure with arbitrary large anti-invariant cohomology. \newline
Let $\Sigma_g$ be a compact Riemann surface of genus $g\geq 2$. On the differentiable product $X=\Sigma_g\times \Sigma_g$, denote by $J$ the complex product structure. Let $\T^2=\R^2\slash\Z^2$ be the real 
$2$-torus, where we indicate with $(t_1,t_2)$ global coordinates on $\R^2$ and let $f:X\to \R$ be a smooth positive non constant function. Let $M=X\times \T^2$. 
Define $\mathcal{J}\in\hbox{\rm End}(TM)$ by setting 
$$
\mathcal{J}(V,a\frac{\partial}{\partial t_1}+b\frac{\partial}{\partial t_2})
=(JV,-\frac{b}{f}\frac{\partial}{\partial t_1}+ fa\frac{\partial}{\partial t_2})
$$
Then, we have the following
\begin{prop}\label{large-anti-invariant}
$\mathcal{J}$ is a non integrable almost complex structure on $M=X\times \T^2$ such that 
$$
h_{\mathcal{J}}^-(M)\geq 2g^2.
$$
\end{prop}
\begin{proof}
It is immediate to check that $\mathcal{J}^2=-\hbox{\rm id}$. Let $p\in X$ such that $df(p)\neq 0$ and let 
$(z_1=x_1+iy_1,z_2=x_2+iy_2)$
be local 
holomorphic coordinates on $X$ around $p$. We may assume that $\frac{\partial}{\partial z_1}f(p)\neq 0$. We have:
$$
\begin{array}{ll}
N_{\mathcal{J}}(\frac{\partial}{\partial x_1},\frac{\partial}{\partial t_1})&= 
[{\mathcal{J}}\frac{\partial}{\partial x_1},{\mathcal{J}}\frac{\partial}{\partial t_1}]-[\frac{\partial}{\partial x_1},\frac{\partial}{\partial t_1}]-{\mathcal{J}}[{\mathcal{J}}\frac{\partial}{\partial x_1},\frac{\partial}{\partial t_1}]
-{\mathcal{J}}[\frac{\partial}{\partial x_1},{\mathcal{J}}\frac{\partial}{\partial t_1}]\\[5pt]
&=[\frac{\partial}{\partial x_1},f\frac{\partial}{\partial t_1}]-
{\mathcal{J}}[\frac{\partial}{\partial x_1},f\frac{\partial}{\partial t_2}]\\[5pt]
&=f_{x_1}(p)\frac{\partial}{\partial t_t}+f_{y_1}(p)\frac{\partial}{\partial t_2}\neq 0
\end{array}
$$
Denote by $\{\gamma_1,\ldots,\gamma_g\}$, $\{\gamma'_1,\ldots,\gamma'_g\}$, 
respectively be a basis of $H^{1,0}_{\overline{\partial}}$ on the first and on the second copy of $\Sigma_g$, respectively. 
Then
$$
H^{(2,0)}_{\overline{\partial}}(X)\simeq\hbox{\rm Span}_\C\langle \gamma_r\wedge \gamma'_s,\quad 1\leq r,s\leq g\rangle 
$$
and clearly $d(\gamma_r\wedge \gamma'_s)=0$, for every $1\leq r,s\leq g$. Then $h_J^-(X)=2g^2$.
Therefore,
$$
h_{\mathcal{J}}^-(M)\geq 2g^2.
$$
\end{proof}
\begin{rem}
The previous Proposition gives a positive aswer to the question raised in \cite[Question 5.2]{ATZ} where it was asked for examples of non integrable almost complex structures 
$J$ on a compact $2n$-dimensional manifold with $h^-_J(M)>n(n-1)$. 
\end{rem}

\end{document}